\documentclass[11pt,letterpaper]{amsart}
\usepackage{amsfonts, amsmath, amssymb, amscd, amsthm, graphicx,enumerate}
\usepackage{epsfig, color}
\usepackage{hyperref}

\makeatletter
\def\blfootnote{\xdef\@thefnmark{}\@footnotetext}
\makeatother

\newtheorem{thm}{Theorem}[section]
\newtheorem{cor}[thm]{Corollary}

\newtheorem{prop}[thm]{Proposition}
\newtheorem{propdfn}[thm]{Proposition-Definition}

\theoremstyle{definition}
\newtheorem{defn}[thm]{Definition}
\theoremstyle{remark}
\newtheorem{rem}[thm]{Remark}

\newfont{\eufm}{eufm10}

\renewcommand{\phi}{\varphi}

\newcommand{\Aut}{{\rm Aut}}

\def\epsilon{\varepsilon}
\def\phi{\varphi}

\def\hat{\widehat}

\begin{document}

\title{Primitivity rank for random elements in free groups}

\author{Ilya Kapovich}

\address{Department of Mathematics and Statistics, Hunter College of CUNY\newline
  \indent 695 Park Ave, New York, NY 10065
  \newline \indent  {\url{http://math.hunter.cuny.edu/ilyakapo/}}, }
  \email{\tt ik535@hunter.cuny.edu}

\thanks{The author was supported by the individual NSF
  grant DMS-1905641}
  
\makeatletter
\@namedef{subjclassname@2020}{%
	\textup{2020} Mathematics Subject Classification}
\makeatother  


\date{}
\maketitle

\begin{abstract}
For a free group $F_r$ of finite rank $r\ge 2$ and a nontrivial element $w\in F_r$ the \emph{primitivity rank} $\pi(w)$ is the smallest rank of a subgroup $H\le F_r$ such that $w\in H$ and that $w$ is not primitive in $H$ (if no such $H$ exists, one puts $\pi(w)=\infty$). The set of all subgroups of $F_r$ of rank $\pi(w)$ containing $w$ as a non-primitive element is denoted $Crit(w)$. These notions were introduced by Puder in \cite{Pu14}. We prove that there exists an exponentially generic subset $V\subseteq F_r$ such that for every $w\in V$ we have $\pi(w)=r$ and $Crit(w)=\{F_r\}$.
\end{abstract}




\section{Introduction}

Recall that an element $g$ of a free group $F$ is \emph{primitive} in $F$ if $g$ belongs to some free basis of $F$. Let $r\ge 2$ be an integer and let $F_r=F(A)$, where $A=\{a_1,\dots, a_r\}$ be the free group of rank $r$. In \cite{Pu14} Puder introduced the following natural measure of algebraic complexity for elements of $F_r$. For an element $1\ne g\in F_r$ the \emph{primitivity rank} $\pi(g)$ in $F_r$ is defined as the smallest rank of a subgroup $H\le F_r$ containing $w$ as a non-primitive element; if no such $H$ exists put $\pi(g)=\infty$. Puder observed that for $1\ne g\in F_r$ one has $\pi(g)=\infty$ if and only if $g$ is primitive in $F_r$. For any non-primitive nontrivial $g\in F_r$ one obviously has $\pi(w)\le r$. For $1\ne g\in F_r$  the \emph{critical set of $g$}, denoted $Crit(g)$, is the set of all subgroups $H$ of $F_r$ containing $g$ as a non-primitive element and such that $rank(H)=\pi(g)$. It is known that the set $Crit(g)$ is always finite (see \cite{PP15}) and that given $g$ one can algorithmically compute $\pi(g)$ and the set $Crit(g)$. Note that the primitivity rank is obviously automorphically invariant. For any $1\ne g\in F_r$ and $\phi\in \Aut(F_r)$ we have $\pi(g)=\pi(\phi(g))$. Moreover, in this case $Crit(\phi(g))=\phi\left( Crit(g)\right) =\{\phi(H)|H\in Crit(g)\}$. 
Puder~\cite{Pu14} and later Puder-Parzanchevski~\cite{PP15} and Hanany-Puder~\cite{HP20} showed that the primitivity rank and the critical set are closely related to understanding the behavior of the \emph{word measures} on finite symmetric groups $S_N$ corresponding to words in $F_r$. Here for a word $w=w(a_1,\dots, a_r)\in F_r$ the corresponding word measure on $S_N$ is the image under the map $w:(S_N)^r\to S_r$ of the uniform probability distribution on $(S_N)^r$. This connection is illustrated in Corollary~\ref{cor:main} to our main result below. 

However, the primitivity rank and the critical set for an element of $F_r$ are natural algebraic notions that deserve to be studied in more detail in their own right.  Indeed, Klimakov~\cite{Kl} extended the notion of primitivity rank to the context of free algebras of Schreier varieties and obtained some structural results there.
Also, a recent series of papers of Louder and Wilton~\cite{LW18,LW21} shows that for $w\in F_r$ there is a connection between the primitivity rank $\pi(w)$ and subgroup properties of the one-relator group $G_w=\langle a_1,\dots, a_r| w=1\rangle$. In particular they prove that if $r\ge 3$ and $\pi(w)\ge 3$ then the group $G_w$ is coherent, that is all of its finitely generated subgroups are finitely presented.

In the present paper we study the primitivity rank and the critical set for "generic" or "random" freely reduced and cyclically reduced elements in $F_r$. The idea of "genericity" in the context of infinite groups arose in the 1990s in the work of Gromov~\cite{Gromov}, Ol'shanskii~\cite{O92}, Arzhantseva~\cite{AO,A1,A2} and others. The notion of genericity for subsets of free groups that we use in this paper was formalized \cite{KMSS03} when defining generic-case complexity of group-theoretic algorithms. If $Z\subseteq F_r$, a subset $S\subseteq Z$ is \emph{exponentially generic} in $Z$ if
\[
\frac{\#\{w\in S: |w|_A\le n\}}{\#\{w\in Z: |w|_A\le n\}}\to_{n\to\infty} 1
\]
with exponentially fast convergence. See Section~\ref{sect:generic} for more precise definitions and additional details. In this paper we are particularly interested in the cases $Z=F_r$ and $Z=U_r$, where $U_r$ is the set of all cyclically reduced words in $F_r=F(A)$.

Our main result is:

\begin{thm}\label{thm:main}
Let $r\ge 2$ be an integer and let $F_r=F(A)$, where $A=\{a_1,\dots, a_r\}$, be the free group of rank $r$. 

\begin{enumerate}

\item There exists a subset $Y_r\subseteq F_r$ such that $Y_r$ is exponentially generic in $F_r$ and that for every $w\in Y_r$ we have  $\pi(w)=r$ and $Crit(w)=\{F_r\}$. 

\item There exists a subset $Q_r\subseteq U_r$ such that  $Q_r$ is exponentially generic in $U_r$ and that for every $w\in Q_r$ we have  $\pi(w)=r$ and $Crit(w)=\{F_r\}$. 

\item Let $W_n\in F(A)$ be obtained by a simple non-backtracking random walk of length $n$. Then, with probability tending to $1$ exponentially fast as $n\to \infty$, the word $W_n$ has the property that $\pi(W_n)=r$ and $Crit(W_n)=\{F_r\}$.   
\end{enumerate}
\end{thm}

Note that the generic sets $Q_r$ and $Y_r$ in the conclusion of Theorem~\ref{thm:main} can in fact be chosen to be algorithmically recognizable because the conditions $P'(\lambda,\mu,r)$ and $(P'(\lambda,\mu,r))^\circ$ defining these sets in the proof of Theorem~\ref{thm:main} are algorithmic if $\lambda$ and $\mu$ are chosen to be rational numbers.

Puder~\cite[Corollary 8.3]{Pu15} also proved the existence of a generic subset of $F_r$ such that for all elements of that subset $w$ the primitivity rank is equal to $r$. His method of proof was very different from ours and it did not imply that $Crit(w)=\{F_r\}$ for generic elements. 

As an application of Theorem~\ref{thm:main}, combined with the result of Puder and Parzanchevski, we obtain precise asymptotics for the number of fixed points $\#fix(\sigma)$ of a random permutation $\sigma$ in $S_N$ with respect to the "word measure" on the symmetric group $S_N$ defined by a generic word $w\in F_r$:

\begin{cor}\label{cor:main}
Let $r\ge 2$ be an integer and let $F_r=F(A)$, where $A=\{a_1,\dots, a_r\}$. Let $Y_r\subseteq F_r, Q_r\subseteq U_r$ be the exponentially generic subsets provided by Theorem~\ref{thm:main}. Then for every $w\in Y_r$ and every $w\in Q_r$ we have
\[
\mathbb E_w[\#fix(\sigma)] \underset{N\to\infty}{=} \ 1+\frac{1}{N^{r-1}}+O\left(\frac{1}{N^r} \right)
\]
Here the expectation is taken with respect to the \emph{word measure} on the symmetric group $S_N$ defined by the word $w(a_1,\dots,a_r)$, that is $\sigma=w(\sigma_1,\dots,\sigma_r)$, where $\sigma_1,\dots \sigma_r\in S_N$ are chosen uniformly independently at random in $S_N$. 
\end{cor}
\begin{proof}
Puder and Parzanchevski proved~\cite[Theorem 1.8]{PP15} that for every $1\ne w\in F_r$ one has
\[
\mathbb E_w[\#fix(\sigma)] \underset{N\to\infty}{=} \ 1+\frac{|Crit(w)|}{N^{\pi(w)-1}}+O\left(\frac{1}{N^{\pi(w)}} \right).
\] 
The statement of the corollary now follows directly from Theorem~\ref{thm:main}.
\end{proof}

As noted above, the results of Louder and Wilton~\cite{LW18,LW21} imply that if $r\ge 3$ and $w\in F_r$ has $\pi(w)\ge 3$ then the group $G_w=\langle a_1,\dots, a_r| w=1\rangle$ is coherent. Since generic elements $w\in F_r$ have $\pi(w)=r$, it follows that for $r\ge 3$ and a generic $w\in F_r$ the group $G_w$ is coherent. This fact was first proved by Sapir and Spukalova~\cite{SaSp}, and also observed by Louder and Wilton~\cite{LW21}, with a reference to Puder's result on the primitivity rank of generic elements~\cite{Pu15}.  

Louder and Wilton conjectured~\cite{LW18} that if $1\ne w\in F_r$ is not a proper power and has $\pi(w)>2$ then the corresponding one-relator group $G_w$ is word-hyperbolic (note that by \cite[Theorem 1.4]{LW18} all 2-generator subgroups in $G_w$ are free and hence $G_w$ has no Baumslag-Solitar subgroups).  Cashen and Hoffmann~\cite{CH20} obtained some experimental evidence for this conjecture, verifying it for all $w$ with $|w|\le 17$ and $r\le 4$. For $1\ne w\in F_r$ which is not a proper power Louder and Wilton define  \emph{$w$-subgroups} of $F_r$ as maximal with respect to inclusion elements of $Crit(w)$.  They show in \cite{LW18} that for $w$ as above if $\pi(w)=2$ then $w$ has a unique $w$-subgroup. In \cite{CH20} raise the question of whether the a $w$-subgroup is always unique. Note that if $\pi(w)=r$ then obviously $F_r$ is the unique maximal element in $Crit(w)$, and thus, in particular, generic elements $w$ of $F_r$ have unique $w$-subgroups.  Our Theorem~\ref{thm:main} provides a maximally sharpened version of this statement for generic elements, since in that case we conclude that the entire set $Crit(w)$ consists of a single element, namely $F_r$ itself.

As shown in \cite{PP15}, both $\pi(w)$ and the set $Crit(w)$ are algorithmically computable in terms of $w$. If $w\in F_r$ is a nontrivial cyclically reduced word (if we are given a word that is not cyclically reduced, we first cyclically reduce it), we write $w$ on a circle of length $w$ to get an $A$-graph $C_w$. Then we consider all possible "quotients"  of $C_w$ under surjective morphisms of folded $A$-graphs. That amounts to picking a partition of the vertex set of $C_w$, collapsing each element of that partition and then folding. For each resulting $A$-graph $\Gamma$ we check, using Whitehead's algorithm, if the closed path $\gamma_w$ labeled by $w$ at the base-vertex $\ast$ of $\Gamma$ is a primitive element of $\pi_1(\Gamma,\ast)$. We keep those $\Gamma$ for which this path $\gamma_w$ is not primitive in $\pi_1(\Gamma,\ast)$ and take the minimum of the ranks of all such $\pi_1(\Gamma,\ast)$. That minimum equals $\pi(w)$ and the $A$-graphs $\Gamma$ that realize this minimum give us elements of $Crit(w)$. Note, however, that this algorithm sheds no light on the behavior of $\pi(w)$ and $Crit(w)$ for "random" cyclically reduced elements $w\in F_r$. Therefore we use rather different considerations in order to prove Theorem~\ref{thm:main}. The key tools there are provided by the "graph non-readability" conditions for generic elements of $F_r$ obtained by Arzhantseva and Ol'shanskii in \cite{AO,A1}, see Section~\ref{sect:AO} below. These genericity conditions also play a key role in the isomorphism rigidity results for "random" one-relator groups obtained in \cite{KS05,KSS06}.

The concept of the primitivity rank is dual to the notion of "primitivity index" for elements of free groups introduced by Gupta and Kapovich in \cite{GK19}. For $1\ne g\in F_r$ the \emph{primitivity index} $d_{prim}(g, F_r)$ of $g$ in $F_r$ is the smallest index of a subgroup $H\le F_r$ such that $H$ contains $g$ as a primitive element. It is shown in \cite{GK19} that $d_{prim}(g, F_r)$ is always finite and moreover $d_{prim}(g, F_r)\le ||g||_A$, where $||g||_A$ is the cyclically reduced length of $g$ in $F(A)$. One of the main results of \cite{GK19} shows that for a "random" freely reduced $w_n\in F(A)$ of length $n$ the primitivity index $d_{prim}(w_n,F_r)$ is $\ge const \log^{1/3} n$. Thus we see that the primitivity index and primitivity rank of generic elements in $F_r$ exhibit rather different quantitative behavior. However, a certain type of duality in these results is still preserved. They show that for a long "generic" $w\in F_r$ it is "hard" for $w$ to be non-primitive in a subgroup of small rank containing $w$ and it is "hard" for $w$ to be primitive in a subgroup of small index in $F_r$ containing $w$. 

As part (3) of Theorem~\ref{thm:main} shows, we can interpret genericity in the context of this theorem in terms of the element $W_n\in F_r=F(A)$ being produced by a simple non-backtracking random walk of length $n$ on $F_r=F(A)$. It would be interesting to understand if the conclusion of Theorem~\ref{thm:main} holds for other types of random walks on $F_r$. For example, let $\mu: F_r\to [0,1]$ be a discrete probability measure with finite support which generates $F_r$. Let $W_n=s_1\dots s_n \in F_r$ be obtained by a random walk of length $n$ on $F_r$ defined by $\mu$ (where the increments $s_1, s_2, ..$ are chosen using an i.i.d. sequence of random variables, each with distribution $\mu$). Is it then true that with probability tending to $1$ as $n\to\infty$ we have $\pi(W_n)=r$ and $Crit(W_n)=\{F_r\}$? We suspect that the answer should be positive but proving this fact would require establishing a suitable version of Proposition~\ref{prop:AO} below for $W_n$. The original proof of Proposition~\ref{prop:AO} by Arzhantseva and Ol'shanskii deployed heavy duty counting arguments relying on certain entropy lowering considerations that were specific to the simple non-backtracking random walk context. Thus a new approach would be required for further generalizations. 

We are grateful to Doron Puder for bringing to our attention his result from \cite{Pu15} that generic elements of $F_r$ have primitivity rank $r$ and for pointing out that the primitivity rank plays a key role in the work of Louder and Wilton~\cite{LW18,LW21}. We also thank Henry Wilton for pointing us to the work of Cashen and Hoffmann~\cite{CH20}. 

\section{Preliminaries}

\subsection{$A$-graphs.}

For the remainder of this paper, unless specified otherwise, let $r\ge 2$ be an integer, $A=\{a_1,\dots,a_r\}$ and let $F_r=F(A)$ be the free group of rank $r$ with the free basis $A$.

We adopt the same language and convention regarding Stallings subgroup graphs for $F_r$ and $A$-graphs more generally as in \cite{KM02}. Thus an \emph{$A$-graph} is a labelled directed graph $\Gamma$ where every oriented edge $e$ is given a label $\mu(e)\in A^{\pm 1}$ satisfying $\mu(e^{-1})=(\mu(e))^{-1}$.  An edge-path $\gamma$ in $\Gamma$ is then naturally assigned a label $\mu(\gamma)$ which is a word in the alphabet $A^{\pm 1}$. An $A$-graph $\Gamma$ is \emph{folded} if whenever $e_1,e_2$ are two distinct oriented edges of $\Gamma$ with the same initial vertex $o(e_1)=o(e_2)$ then $\mu(e_1)\ne \mu(e_2)$. Thus every vertex in a folded $A$-graph has degree $\le 2r$.

A connected $A$-graph $\Gamma$ with a base-vertex $\ast$ is a \emph{core graph} with respect to $\ast$ if $\Gamma$ equals to the union of all non-backtracking closed edge-paths from $\ast$ to $\ast$. A connected $A$-graph $\Gamma$ is a \emph{core graph} if it is a core graph with respect to each of its vertices.

For a finite $A$-graph $\Gamma$ we denote by $vol(\Gamma)$ the number of topological edges of $\Gamma$ (where for every oriented edge $e$ of $\Gamma$ the unordered pair $e,e^{-1}$ counts as a single topological edge).

Every subgroup $H\le F_r$ is uniquely represented by its \emph{Stallings subgroup graph} $\Gamma_H$, which is a connected folded $A$-graph with a base-vertex $\ast$ such that $\Gamma_H$ is a core graph with respect to $\ast$ and that a freely reduced word $w\in F(A)$ belongs to $H$ if and only if $w$ labels a closed path from $\ast$ to $\ast$ in $\Gamma_H$. In this case the labeling map $\mu$ provides a natural isomorphism between $\pi_1(\Gamma_H,\ast)$ and $H$. A subgroup $H\le F(A)$ is finitely generated if and only if $\Gamma_H$ is finite. Moreover, $H$ has finite index in $F(A)$ if and only if every vertex of $\Gamma_H$ has degree $2r$, and in this case the index $[F(A):H]$ is equal to the number of vertices in $\Gamma_H$. 

For $H=F_r$ the corresponding Stallings subgroup graph is the \emph{$A$-rose} $R_A$, which is a wedge of $r$-loop edges, labelled $a_1,\dots, a_r$, wedged at a single vertex $x_0$. For an arbitrary $H\le F(A)$ one can obtain $\Gamma_H$ by first taking the covering space $(\hat R_A, \ast)$ of $(R_A,x_0)$ corresponding to $H\le \pi_1(R_A,x_0)$ and then taking the topological core of this covering that is the smallest connected subgraph of $\hat R_A$ containing $\ast$ whose inclusion into $\hat R_A$ is a homotopy equivalence. 

We refer the reader to ~\cite{St83,KM02,KS05} for additional background on $A$-graphs and Stallings subgroup graphs.

For an element $g\in F_r$ we denote by $|g|_A$ the freely reduced length of $g$ with respect to $A$ and we denote by $||g||_A$ the cyclically reduced length of $g$ with respect to $A$. For a word $w$ over $A^{\pm 1}$ (not necessarily freely reduced) we denote by $|w|$ the length of $w$. Similarly, for an edge-path $\gamma$ in an $A$-graph $\Gamma$, we denote by $|\gamma$ the length of $\gamma$. 

Let $\Gamma$ be a finite connected $A$-graph with a base-vertex $\ast$. The \emph{maximal arcs} in $\Gamma$ are closures of the connected components of the space obtained by removing from $\Gamma$ all vertices of degree $\ge 3$ and the vertex $\ast$. It's easy to check that if $\Gamma$ is a finite connected $A$-graph with $\pi_1(\Gamma)$ of rank $m$ then $\Gamma$ has $\le 3m$ maximal arcs.

\subsection{Genericity}\label{sect:generic}

The definitions in this section follow \cite{KMSS03, KS05}. See \cite{Ka20} and \cite[Ch. 1]{Gagta}  for more general treatment of genericity. 

For a sequence $(x_n)_{n\ge n_0}$ of real numbers and $x\in \mathbb R$ we say that $\lim_{n\to \infty} x_n=x$ \emph{with exponentially fast convergence} if there exist constants $C>0$ and $0<\sigma<1$ such that for all $n\ge n_0$ we have 
\[
 |x_n-x|\le C\sigma^n.
\]

Let $S\subseteq F(A)$. For $n\ge 0$ we denote by $\rho_n(S)$ the number of elements $w\in F(A)$ with $|w|_A\le n$ and we denote by $\gamma_n(S)$ the number of elements $w\in F(A)$ with $|w|_A= n$.

\begin{defn}[Exponentially generic subsets]
Let $Z\subseteq F_r$ be a subset such that there is $n_0\ge 1$ such that for all $n\ge n_0$ we have $\rho_n(S)\ne 0$.
\begin{enumerate}
\item A subset $S\subseteq Z$ is \emph{exponentially negligible} in $Z$ if 
\[
\lim_{n\to\infty} \frac{\rho_n(S)}{\rho_n(Z)}=0
\]
and the convergence in this limit is exponentially fast. 

\item A subset $S\subseteq Z$ is \emph{exponentially generic} in $Z$ if $Z\setminus S$ is exponentially negligible in $Z$.
\end{enumerate}
\end{defn}

Recall that $U_r$ denotes the set of all cyclically reduced words in $F_r=F(A)$.

Both $n$-balls and $n$-spheres in $F_r$ and $U_r$ grow like $(2r-1)^n$. This fact has the following useful consequences, see \cite[Lemma 6.1]{KSS06}

\begin{prop}\label{prop:use} The following hold:
\begin{enumerate}
\item[(a)] Let $S\subseteq F_r$. Then $S$ is exponentially negligible in $F_r$ if and only if $\lim_{n\to\infty} \frac{\gamma_n(S)}{(2r-1)^n}=0$ with exponentially fast convergence.
\item[(b)] Let $S\subseteq U_r$. Then $S$ is exponentially negligible in $U_r$ if and only if $\lim_{n\to\infty} \frac{\gamma_n(S)}{(2r-1)^n}=0$ with exponentially fast convergence.
\item[(c)]  Let $S\subseteq F_r$. Then $S$ is exponentially generic in $F_r$ if and only if $\lim_{n\to\infty} \frac{\gamma_n(S)}{\gamma_n(F_r)}=1$ with exponentially fast convergence.
\item[(d)]  Let $S\subseteq U_r$. Then $S$ is exponentially generic in $F_r$ if and only if $\lim_{n\to\infty} \frac{\gamma_n(S)}{\gamma_n(U_r)}=1$ with exponentially fast convergence.
\end{enumerate}
\end{prop}

We also need the following statement, see \cite[Proposition 6.2]{KSS06}

\begin{propdfn}\label{circ}
Let $Z\subseteq U_r$. Let $Z^\circ$ denote the set of all freely reduced words in $F_r$ whose cyclically reduced forms belong to $Z$.

\begin{enumerate}
\item If $Z$ is exponentially negligible in $U_r$ then $Z^\circ$ is exponentially negligible in $F_r$.
\item If $Z$ is exponentially generic in $U_r$ then $Z^\circ$ is exponentially generic in $F_r$.
\end{enumerate}
\end{propdfn}

\section{The genericity condition}\label{sect:AO}

Recall that $r\ge 2$ is fixed, $A=\{a_1,\dots, a_r\}$ and $F_r=F(A)$. We will need the following genericity condition introduced by Arzhantseva and Ol'shanskii in \cite{AO}.

\begin{defn}\cite{AO}
  Let $0< \mu\le 1$ be a real number. A nontrivial freely reduced word $w$ in
  $F(A)=F(a_1,\dots, a_r)$  is called
  \emph{$\mu$-readable} if there exists a connected folded $A$-graph
  $\Gamma$ such that:

\begin{enumerate}
\item The number of edges in $\Gamma$ is at most $\mu |w|$.
\item The free group $\pi_1(\Gamma)$ has rank at most $r-1$.
\item There exists a reduced path in $\Gamma$ with label $w$.
\end{enumerate}
\end{defn}

\begin{defn}\cite{A1}
  Let $0< \mu\le 1$ be a real number and let $L\ge 2$ be an integer. A
  nontrivial freely reduced word $w$ in $F(A)$ is called
  \emph{$(\mu,L)$-readable} if there exists a connected folded
  $A$-graph $\Gamma$ such that:

\begin{enumerate}
\item The number of edges in $\Gamma$ is at most $\mu |w|$.
\item The free group $\pi_1(\Gamma)$ has rank at most $L$.
\item There is a path in $\Gamma$ with label $w$.
\item The graph $\Gamma$ has at least one vertex of degree $< 2r$.
\end{enumerate}
\end{defn}

\begin{defn}\label{LML}
  Let $0< \mu\le 1$ be a real number, let $L\ge 2$ be an integer and let $0<\lambda<1$ be a real number such that 
  \[
  \lambda
  \le \frac{\mu }{15L+3\mu}\le \frac{\mu }{15r+3\mu}< 1/6.
  \] 
  We will
  say that a cyclically reduced word $w$ in $F(A)$ satisfies the \emph{$(\lambda, \mu, L)$-condition}
  if:

\begin{enumerate}
\item The (symmetrized closure) of the word $w$ satisfies the $C'(\lambda)$ small
  cancellation condition.
\item The word $w$ is not a proper power in $F(A)$.
\item If $w'$ is a subword of a cyclic permutation of $w$ and
  $|w'|\ge |w_i|/2$ then $w'$ is not 
  $\mu$-readable and not $(\mu,L)$-readable.
\end{enumerate}

We denote by $P(\lambda, \mu,L)$ the set of all cyclically reduced words $w\in F_r$ satisfying the $(\lambda,\mu,L)$-condition.

\end{defn}
We refer the reader to \cite{LS} for the basic definitions and
background information regarding small cancellation theory.

Note that (3) in Definition~\ref{LML} implies that if $w\in P(\lambda, \mu,L)$ then $w$ is not 
  $\mu$-readable and not $(\mu,L)$-readable.

The results of Arzhantseva and Ol'shanskii~\cite{AO,A1} impliy the following result (c.f. \cite[Theorem 4.6]{KS05}).

\begin{prop}\label{prop:AO}
Let $r\ge 2$ and $F_r=F(a_1,\dots, a_r)$. Let $\lambda, \mu, L$ be
  as in Definition~\ref{LML}.

  Then the set $P(\lambda, \mu, L)$ is exponentially generic in $U_r$.
\end{prop}

\begin{rem}
Note that it is fairly easy to see that for any $\lambda\in (0,1)$ the set all of non-proper power words satisfying $C'(\lambda)$ is exponentially generic in $U_r$. The main substance and power of Proposition~\ref{prop:AO} consists in verifying genericity of the non-readability conditions.

Note also that we don't actually need the full strength of the $(\lambda, \mu, L)$-condition for the proofs in this paper. In particular, we don't use the "not a proper power" assumption on $w$. Also, in part (3) of Definition~\ref{LML} we only really need non-readability of $w$ itself rather than of its sufficiently long subwords. Nevertheless, we state the $(\lambda, \mu, L)$-condition in its original stronger form, as it was defined by Arzhantseva and Ol'shanskii, for which they established Proposition~\ref{prop:AO}.
\end{rem}

We also need the following slightly more restricted version of the set $P(\lambda, \mu, L)$.

\begin{defn}
Let $r\ge 2$ and $F_r=F(a_1,\dots, a_r)$. Let $\lambda, \mu, L$ be
  as in Definition~\ref{LML}.

Denote by $P'(\lambda, \mu, L)$ the set of all words $w\in P(\lambda, \mu, L)$ such that every freely reduced word of length 2 in $F_r$ occurs as a subword of $w$. 
\end{defn}

\begin{prop}\label{prop:AOm}
Let $r\ge 2$ and $F_r=F(a_1,\dots, a_r)$. Let $\lambda, \mu, L$ be
  as in Definition~\ref{LML}.

  Then the set $P'(\lambda, \mu, L)$ is exponentially generic in $U_r$.
\end{prop}

\begin{proof}
For a given freely reduced word $z$ of length 2, the set of all words in $U_r$ containing $z$ as a subword is exponentially generic in $U_r$; see, for example~\cite[Proposition~6.3]{KSS06}.
Since the intersection of finitely many exponentially generic subsets of $U_r$ is again exponentially generic in $U_r$, the conclusion of the proposition now follows from Proposition~\ref{prop:AO}.
\end{proof} 

\section{Primitivity rank of generic elements}

\begin{thm}\label{thm:primrank}
Let $r\ge 2$. Let $0<\lambda, \mu<1$ be such that 
\[
  \lambda
  \le \frac{\mu }{15L+3\mu}\le \frac{\mu }{15r+3\mu}< 1/6
\] 
and that
\[
\lambda<\frac{\mu}{3r}.
\]
Then for every nontrivial cyclically reduced word $w\in P'(\lambda, \mu,r)$ we have 
\[
\pi(w)=r \quad \text{ and } \quad Crit(w)=\{F_r\}.
\]
\end{thm}

\begin{proof}

Let $1\ne w\in P'(\lambda, \mu,r)$ where $\lambda,\mu$ are as in the assumptions of the theorem.

a) Note first that $w$ is not primitive in $F_r$ because $w$ contains all freely reduced words in $F_r=F(A)$ of length 2 as subwords and hence the Whitehead graph of the cyclic word  corresponding to $w$ is complete~\cite{Wh36}. Hence $\pi(w)\ne \infty$ and so $\pi(w)\le r$. 

b) We claim that $\pi(w)=r$. Indeed, suppose not. Thus $\pi(w)<r$. Take a subgroup $H\in Crit(w)$. 

Thus $H\le F_r$ is a subgroup of the smallest possible rank containing $w$ as a non-primitive element and $rank(H)=\pi(w)<r$.

Let $\Gamma_H$ be the Stallings subgroup graph for $H$ and let $\ast$ be the base-vertex of $\Gamma_H$. Then the cyclically reduced word $w$ is readable along a closed path $\gamma_w$ from $\ast$ to $\ast$ in $\Gamma_H$. The minimality assumption on the rank of $H$ implies that $\gamma_w$ crosses every topological edge of $\Gamma_H$. Moreover, since $\Gamma_H$ is folded and $w$ is cyclically reduced, the vertex $\ast$ has degree $>1$ in $\Gamma_H$, that is $\Gamma_H$ is a folded connected finite core $A$-graph.

Put $m=rank(H)<r$.

The graph $\Gamma_H$ has $\le 3m$ maximal arcs (recall that we view $\ast$ as an endpoint of maximal arcs even if $\ast$ has degree 2 in $\Gamma_H$).

The case $vol(\Gamma_H)\le \mu |w|$ is impossible since $w$ is readable in the graph $\Gamma_H$ of rank $\le r-1$ and since by assumption $w$ satisfies the $(\lambda,\mu,r)$-condition.

Therefore $vol(\Gamma_H)\ge \mu |w|$. Let $\alpha$ be the longest maximal arc of $\Gamma_H$. Since $\Gamma_H$ has $\le 3m$ maximal arcs, we have 
\[
|\alpha|\ge vol(\Gamma_H)/3m\ge \mu |w|/3m\ge \mu |w|/3r >\lambda |w|.
\]

The path $\gamma_w$ crosses over the arc $\alpha$ (in some direction) at least once, and the $C'(\lambda)$ assumption on $w$ implies that it does so exactly once. Since $\gamma_w$ is a closed path at $\ast$ in $\Gamma_H$, it follows that $\alpha$ is a non-separating arc in $\Gamma_H$ and that $\gamma_w$ represents a primitive element in $\pi_1(\Gamma_H,\ast)$. Hence $w$ is primitive in $H$, yielding a contradiction with our assumption about $w$ and $H$.

Thus indeed $\pi(w)=r$, as claimed.

c) We now claim that $Crit(w)=\{F_r\}$.

We have already seen in a) that $w$ is not primitive in $F_r$ and we have just proved that $\pi(w)=r$, so that $F_r\in Crit(w)$.

Suppose that $Crit(w)\ne \{F_r\}$. Then there exists a subgroup $H\in Crit(w)$, such that $H\ne F_r$. Thus $H$ contains $w$ as a non-primitive element and $rank(H)=r$. Since $H\ne F_r$, it follows that $H$ has infinite index in $F_r$. 

Again, let $\Gamma_H$ be the Stallings subgroup graph for $H$ with base-vertex $\ast$. Since $w\in H$, there is a closed path from $\ast$ to $\ast$ labelled by $w$. As in b), we see that $\Gamma_H$ is a finite connected  folded core $A$-graph. Also, the minimality assumption on the rank of $H$ implies that $\gamma_w$ crosses every edge of $\Gamma_H$. Since $[F_r:H]=\infty$, the graph $\Gamma_H$ has a vertex of degree $<2r$.

Since $w$ is readable in $\Gamma_H$ but $w$ is not $(\mu,r)$-readable by our assumption that $w\in P'(\lambda,\mu,r)$, it follows that the case $vol(\Gamma_H)\le \mu|w|$ is impossible. Hence $vol(\Gamma_H)\ge \mu|w|$.

Since $\pi_1(\Gamma_H)$ has rank $r$, the  graph $\Gamma_H$ has $\le 3r$ maximal arcs. Let $\alpha$ be the longest maximal arc of $\Gamma_H$. Then 
\[
|\alpha|\ge vol(\Gamma_H)/3r\ge \mu |w|/3r>\lambda |w|.
\]

As in b), the path $\gamma_w$ crosses the arc $\alpha$ at least once in some direction, and the $C'(\lambda)$ condition implies that it crosses $\alpha$ exactly once. Since $\gamma_w$ is a closed path, it again follows that $\alpha$ is a non-separating arc in $\Gamma_H$ and hence $\gamma_w$ is primitive in $\pi_1(\Gamma_H,\ast)$. Therefore $w$ is primitive in $H$, yielding a contradiction with the choice of $H$.

Thus $|Crit(w)|=1$, as required.
\end{proof}

We can now prove Theorem~\ref{thm:main} from the Introduction:

\begin{thm}\label{thm:main1}
Let $r\ge 2$ be an integer and let $F_r=F(A)$, where $A=\{a_1,\dots, a_r\}$, be the free group of rank $r$. 

\begin{enumerate}

\item There exists a subset $Y_r\subseteq F_r$ such that $Y_r$ is exponentially generic in $F_r$ and that for every $w\in Y_r$ we have  $\pi(w)=r$ and $Crit(w)=\{F_r\}$. 

\item There exists a subset $Q_r\subseteq U_r$ such that  $Q_r$ is exponentially generic in $U_r$ and that for every $w\in Q_r$ we have  $\pi(w)=r$ and $Crit(w)=\{F_r\}$. 

\item Let $W_n\in F(A)$ be obtained by a simple non-backtracking random walk of length $n$. Then, with probability tending to $1$ exponentially fast as $n\to \infty$, the word $W_n$ has the property that $\pi(W_n)=r$ and $Crit(W_n)=\{F_r\}$.   
\end{enumerate}
\end{thm}

\begin{proof}

Choose $\lambda,\mu$ as in the assumptions of Theorem~\ref{thm:primrank} and put $Q_r=P'(\lambda, \mu,r)$. Then $Q_r$ is exponentially generic in $U_r$ by Proposition~\ref{prop:AOm}. Then $Q_r$ satisfies the requirements of part (2) by Theorem~\ref{thm:primrank}.

Now put $Y_r=Q_r^\circ$. Then $Y_r$ is exponentially generic in $F_r$ by Proposition-Definition~\ref{circ}. Every element $w$ of $Y_r$ is conjugate in $F_r$ to some element $w'$ of $Q_r$ in $F_r$, that is $w=gw'g^{-1}$ for some $g\in F_r$. Then $\pi(w)=\pi(w')$ and $Crit(w)=gCrit(w')g^{-1}=\{gF_rg^{-1}\}=\{F_r\}$. Thus the requirements of (1) hold for $Y_r$.

Note that the simple non-backtracking random walk $W_n$ of length $n$ on $F(A)$ induces the uniform probability distribution on the $n$-sphere in $F(A)$. 
Therefore (3) directly follows from (1) in view of part 3 of Proposition~\ref{prop:use}.

\end{proof}

\vspace{1cm}


\begin{thebibliography}{999}



\bibitem{AO} G.~Arzhantseva and A.~Olshanskii, \emph{Genericity of the
    class of groups in which subgroups with a lesser number of
    generators are free,} (Russian) Mat. Zametki \textbf{59} (1996),
  no. 4, 489--496

\bibitem{A1} G.~Arzhantseva, \emph{On groups in which subgroups with a
    fixed number of generators are free,}(Russian) Fundam. Prikl. mat.
  \textbf{3} (1997), no. 3, 675--683
  
\bibitem{A2} G.~Arzhantseva, \emph{Generic properties of finitely
    presented groups and Howson's theorem,} Comm. Algebra \textbf{26}
  (1998), no. 11, 3783--3792
  
  
\bibitem{CH20}
C. H. Cashen and C. Hoffmann,
\emph{Short, highly imprimitive words yield hyperbolic one-relator groups}, preprint, 2020,  arXiv:2006.15923
  
   
\bibitem{Gagta}   
 F. Bassino, I. Kapovich, M. Lohrey, A. Miasnikov, C. Nicaud, A. Nikolaev, I. Rivin, V. Shpilrain, A. Ushakov, P. Weil, \emph{Complexity and randomness in group theory -- GAGTA book 1},  De Gruyter, Berlin, 2020, ISBN: 978-3-11-066702-8


\bibitem{Gromov}
M. Gromov, \emph{Asymptotic invariants of infinite groups}, in: Geometric Group Theory, Vol. 2, Sussex, 1991, Cambridge Univ. Press, Cambridge, 1993, pp. 1--295

\bibitem{GK19}
N. Gupta, and I. Kapovich, \emph{The primitivity index function for a free group, and untangling closed curves on hyperbolic surfaces. With an appendix by Khalid Bou-Rabee}, Math. Proc. Cambridge Philos. Soc. \textbf{166} (2019), no. 1, 83--121

\bibitem{HP20}
L. Hanany and D. Puder,
\emph{Word Measures on Symmetric Groups,} preprint,  September 2020, arXiv:2009.00897  


\bibitem{Ka20}
I. Kapovich, \emph{Musings on generic-case complexity}, Elementary theory of groups and group rings, and related topics, 135--148, De Gruyter Proc. Math., De Gruyter, Berlin, 2020


\bibitem{KM02}
I. Kapovich and A. Myasnikov, \emph{Stallings foldings and subgroups of free groups}, J. Algebra \textbf{248} (2002), no. 2, 608--668


\bibitem{KS05}
I. Kapovich and P. Schupp, \emph{Genericity, the Arzhantseva-Ol'shanskii method and the isomorphism problem for one-relator groups,} Math. Ann. \textbf{331} (2005), no. 1, 1--19

\bibitem{KMSS03}
I. Kapovich, A. Myasnikov, P. Schupp, and V. Shpilrain, \emph{Generic-case complexity, decision problems in group theory, and random walks.} J. Algebra \textbf{264} (2003), no. 2, 665--694

\bibitem{KSS06}
I. Kapovich, P. Schupp, and V. Shpilrain, \emph{Generic properties of Whitehead's algorithm and isomorphism rigidity of random one-relator groups.} Pacific J. Math. \textbf{223} (2006), no. 1, 113--140

\bibitem{Kl}
A. Klimakov, 
\emph{Primitivity rank of elements of free algebras of Schreier varieties},
J. Algebra Appl. \textbf{15} (2016), no. 2, 1650036

\bibitem{LW18}
L. Louder and H. Wilton,
\emph{Negative immersions for one-relator groups}, Duke Mathematical Journal, to appear; arXiv:1803.02671


\bibitem{LW21}
L. Louder and H. Wilton,
\emph{Uniform negative immersions and the coherence of one-relator groups}, preprint, 2021, arXiv:2107.08911

\bibitem{LS}
R. Lyndon and P. E. Schupp,
\emph{Combinatorial group theory.} Reprint of the 1977 edition. Classics in Mathematics. Springer-Verlag, Berlin, 2001; ISBN: 3-540-41158-5



\bibitem{MT18}
J. Maher, and J. Tiozzo, \emph{Random walks on weakly hyperbolic groups.} J. Reine Angew. Math. \textbf{742} (2018), 187--239

\bibitem{O92}
A. Yu. Ol'shanskii, \emph{Almost every group is hyperbolic}, Internat. J. Algebra Comput. \textbf{2} (1992) 1--17

\bibitem{Pu14}
D. Puder, \emph{Primitive words, free factors and measure preservation}. Israel J. Math. \textbf{201} (2014), no. 1, 25--73

\bibitem{Pu15}
D. Puder, \emph{Expansion of random graphs: new proofs, new results. Invent. Math. 201 (2015), no. 3, 845?908}

\bibitem{PP15}
D. Puder and O. Parzanchevski, 
\emph{Measure preserving words are primitive.} J. Amer. Math. Soc. \textbf{28} (2015), no. 1, 63--97 

\bibitem{SaSp}
M. Sapir and I. Spakulov\'a,
\emph{Almost all one-relator groups with at least three generators are residually finite}, J. Eur. Math. Soc. (JEMS) \textbf{13} (2011), no. 2, 331--343

\bibitem{St83}
J. R. Stallings, \emph{Topology of finite graphs}, Invent. Math. \textbf{71} (1983), no. 3, 551--565

\bibitem{Wh36}
J. H. C. Whitehead, \emph{On equivalent sets of elements in a free group}, Ann. of Math. (2) \textbf{37} (1936), no. 4, 782--800

\end{thebibliography}
\end{document}